\newcommand{\dataversione}{\today}

\documentclass[11pt,reqno,a4paper]{amsart}

\usepackage{mathrsfs, graphics}
\usepackage{amssymb}
\usepackage{amsfonts}
\usepackage{enumitem} 
\usepackage{hyperref}
\usepackage{graphicx, epstopdf}\graphicspath{{./immaginiprova/}}
\usepackage{color}
\usepackage[mathscr]{euscript}
\theoremstyle{plain} 

\theoremstyle{definition}

\numberwithin{equation}{section}

\newtheoremstyle{mytheorem}
{}
{}
{\it}
{\parindent}
{\bf}
{.}
{ }
{\thmnumber{#2.~}\thmname{#1}\thmnote{~\rm#3}}

\newtheoremstyle{myremark}
{}
{}
{\rm}
{\parindent}
{\bf}
{.}
{ }
{\thmnumber{#2.~}\thmname{#1}\thmnote{~\rm#3}}

\newtheoremstyle{myparagraph}
{}
{}
{\rm}
{\parindent}
{\bf}
{.}
{ }
{\thmnumber{#2.~}\thmname{#1}\thmnote{#3}}

\theoremstyle{mytheorem}

\newtheorem{theorem}[subsection]{Theorem}
\newtheorem{lemma}[subsection]{Lemma}

\newtheorem{corollary}[subsection]{Corollary}
\newtheorem{proposition}[subsection]{Proposition}

\theoremstyle{myremark}
\newtheorem{remark}[subsection]{Remark}

\theoremstyle{myparagraph}

\newtheorem*{parag*}{}

\makeatletter
\def\@secnumfont{\sc}
\def\section{\@startsection{section}{1}%
\z@{1.5\linespacing\@plus .2\linespacing}{.7\linespacing}%
{\normalfont\sc\centering}}
\makeatother

\makeatletter
\def\ps@headings{\ps@empty
 \def\@evenhead{%
  \setTrue{runhead}%
  \normalfont\footnotesize
  \rlap{\thepage}\hfil
  \def\thanks{\protect\thanks@warning}%
  \leftmark{}{}\hfil}%
 \def\@oddhead{%
  \setTrue{runhead}%
  \normalfont\footnotesize\hfil
  \def\thanks{\protect\thanks@warning}%
  \rightmark{}{}\hfil \llap{\thepage}}%
\let\@mkboth\markboth}
\makeatother
\makeatletter
\renewenvironment{proof}[1][\proofname]{\par
  \pushQED{\qed}%
  \normalfont \topsep6\p@\@plus6\p@\relax
  \trivlist
  \itemindent\normalparindent
  \item[\hskip\labelsep
    \bfseries
    #1\@addpunct{.}]\ignorespaces
}{%
  \popQED\endtrivlist\@endpefalse
}
\providecommand{\proofname}{Proof}
\makeatother
\newcommand{\R}{\mathbb{R}}
\newcommand{\N}{\mathbb{N}}

\newcommand{\Haus}{\mathscr{H}}
\newcommand{\Leb}{\mathscr{L}}

\newcommand{\Tan}{\mathrm{Tan}}

\newcommand{\wrt}{w.r.t.\ }

\newcommand{\dV}{d_V\kern-1pt}

\newcommand{\trait}[3]{\vrule width #1ex height #2ex depth #3ex}

\newcommand{\trace}{\mathchoice%
  {\mathbin{\trait{.12}{1.2}{.03}\trait{.8}{0.09}{0.03}}}
  {\mathbin{\trait{.12}{1.2}{.03}\trait{.8}{0.09}{0.03}}}
  {\mathbin{\hskip.15ex\trait{.09}{.84}{0.02}\trait{.56}{.07}{.02}}\hskip.15ex}
  {\mathbin{\trait{.07}{.6}{.01}\trait{.4}{.06}{.01}}}}

\makeatletter
\@addtoreset{footnote}{section}
\makeatother

\newenvironment{itemizeb}
{\begin{itemize}\itemsep=2pt}{\end{itemize}}

\begin{document}

	%
\pagestyle{empty}
\pagestyle{myheadings}
\markboth%
{\underline{\centerline{\hfill\footnotesize%
\textsc{Andrea Marchese}%
\vphantom{,}\hfill}}}%
{\underline{\centerline{\hfill\footnotesize%
\textsc{Covering singular measures}%
\vphantom{,}\hfill}}}

	%
\thispagestyle{empty}

~\vskip -1.1 cm

	%
{\footnotesize\noindent
[version:~\dataversione]%
\hfill
}

\vspace{1.7 cm}

	%
{\Large\sl\centering
A covering theorem for singular measures\\ in the Euclidean space
\\
}
\vspace{.6 cm}

	%
\centerline{\sc Andrea Marchese}

\vspace{.8 cm}

{\rightskip 1 cm
\leftskip 1 cm
\parindent 0 pt
\footnotesize

	%
{\sc Abstract.} We prove that for any singular measure $\mu$ on $\mathbb{R}^n$ it is possible to cover $\mu$-almost every point with $n$ families of Lipschitz slabs of arbitrarily small total width. More precisely, up to a rotation, for every $\delta>0$ there are $n$ countable families of $1$-Lipschitz functions $\{f_i^1\}_{i\in\mathbb{N}},\ldots, \{f_i^n\}_{i\in\mathbb{N}},$
$f_i^j:\{x_j=0\}\subset\mathbb{R}^n\to\mathbb{R}$, and $n$ sequences of positive real numbers 
$\{\varepsilon_i^1\}_{i\in\mathbb{N}},\ldots, \{\varepsilon_i^n\}_{i\in\mathbb{N}}$
such that, denoting $\hat x_j$ the orthogonal projection of the point $x$ onto $\{x_j=0\}$ and
$$I_i^j:=\{x=(x_1,\ldots,x_n)\in \mathbb{R}^n:f_i^j(\hat x_j)-\varepsilon_i^j< x_j< f_i^j(\hat x_j)+\varepsilon_i^j\},$$
it holds $\sum_{i,j}\varepsilon_i^j\leq \delta$ and $\mu(\mathbb{R}^n\setminus\bigcup_{i,j}I_i^j)=0.$

We apply this result to show that, if $\mu$ is not absolutely continuous, it is possible to approximate the identity with a sequence $g_h$ of smooth equi-Lipschitz maps satisfying 
$$\limsup_{h\to\infty}\int_{\mathbb{R}^n}{\rm{det}}(\nabla g_h) d\mu<\mu(\mathbb{R}^n).$$
From this, we deduce a simple proof of the fact that every top-dimensional Ambrosio-Kirchheim metric current in $\R^n$ is a Federer-Fleming flat chain.


\par
}

%
%

\section{Introduction}
Fix an orthonormal basis $(e_1,\ldots,e_n)$ of $\R^n$. For $j=1,\ldots,n$, and for $x\in\R^n$, we denote $\hat x_j\in\R^{n-1}$ the orthogonal projection of $x$ onto $\{x_j=0\}$. Given a function $f:\R^{n-1}\to \R$ and $\varepsilon>0$ we consider the set 
$$I_\varepsilon^j(f):=\{x=(x_1,\ldots,x_n)\in \R^n:f(\hat x_j)-\varepsilon< x_j< f(\hat x_j)+\varepsilon\}$$
and we call it the \emph{open slab} around $f$, of \emph{width} $\varepsilon$, in direction $e_j$.

Given a family $\mathcal{F}$ of slabs, we denote $w(\mathcal{F})$ its \emph{total width}, i.e. the sum of the widths of the corresponding slabs.
For a fixed sequence $\{(f_i^j,\varepsilon_i^j)\}_{(i,j)\in\N\times\{0,\ldots,n\}}$ with $f_i^j:\R^{n-1}\to\R^n$ and $\varepsilon_i^j$ positive real numbers, we use the short notation $I_i^j$ to denote the slab $I_{\varepsilon_i}^j(f_i^j)$.

Given a measure $\mu$ on $\R^n$ and a Borel function $\rho:\R^n\to\R^n$ we denote by $\rho_\sharp\mu$ the push forward of $\mu$ via $\rho$, i.e. the measure defined by
$$\rho_\sharp\mu(A):=\mu(\rho^{-1}(A)),$$
for every Borel set $A$. 

The main result of this note is the following theorem. 
\begin{theorem}\label{main}
Let $\mu$ be a finite Borel measure on $\R^n$, $n\geq 2$, which is singular with respect to the Lebesgue measure. Then there exists a rotation $\rho:\R^n\to\R^n$ with the following property. For every $\delta>0$ there is a sequence $\{(f_i^j,\varepsilon_i^j)\}_{(i,j)\in\N\times\{0,\ldots,n\}}$ where $f_i^j:\R^{n-1}\to\R$ are $1$-Lipschitz functions and $\varepsilon_i^j$ are positive real numbers such that the family of slabs $\mathcal{F}:=\{I_i^j\}_{(i,j)\in\N\times\{0,\ldots,n\}}$ has total width $w(\mathcal{F})\leq\delta$ and
$$\rho_\sharp\mu\left(\R^n\setminus\bigcup_{i,j}I_i^j\right)=0.$$
\end{theorem}

\begin{remark}
\begin{itemize}

\item[(i)] For $n=2$, Theorem \ref{main} is a straightforward consequence of a covering result for nullsets, which will appear in \cite{ACP3}. Actually a weaker version  of such covering result, proved in \cite{ACP1} and \cite{ACP2} (i.e. for compact nullsets), would also suffice to our purpose.
\item[(ii)] For $n>2$, the theorem follows from a stronger result, announced by M. Cs\"ornyei and P. Jones (see \cite{CJ}). The proof we present here is considerably simpler. We remark that all the ``ingredients'' for the proof were already available in the literature, indeed the proof is achieved combining a corollary of the main result in \cite{DPR} with some results obtained in \cite{AM} and some important ideas from \cite{ACP3}, also used in \cite{Ma}.
\item[(iii)] In Lemma \ref{lemma}, we prove that the set of rotations $\rho$ for which the conclusion of Theorem \ref{main} holds, has full measure in $SO(n)$. In particular, one can chose a rotation which is arbitrarily close to the identity map, and then reparametrize the graphs of the Lipschitz functions $f_i^j$ with respect to the tilted coordinates. Hence one can get rid of the rotation $\rho$ in the statement, at the price of increasing the Lipschitz constant of an arbitrarily small quantity. 
\item[(iv)] This result can be used (see \cite[Chapter 4]{Ma}) to prove the weak converse of Rademacher's theorem, namely that for every singular measure $\mu$ on $\R^n$ there exist a Lipschitz function $f:\R^n\to\R$ which is $\mu$-a.e. non-differentiable. This was later 
improved in \cite{MS}, where it is proved that it is possible to find a Lipschitz function which admits any pointwise prescribed blowup, provided it is linear along the decomposability bundle of $\mu$ (see \S \ref{s:decomp}), at every point except for a set of arbitrarily small measure $\mu$. The converse of Rademacher's theorem has also important consequences in the study of Lipschitz differentiability metric measure spaces and of spaces with Ricci curvature bounded from below (see e.g. \cite{B,KM,DPMR,GP}).
\end{itemize}
\end{remark}

In \S\ref{currents}, we apply Theorem \ref{main} to obtain a simple proof of the case $k=n$ of the ``flat chain conjecture'' stated in \cite[Section 11]{AK}. Namely we prove that for any Ambrosio-Kirchheim metric current $T$ of dimension $n$ in $\R^n$, the measure $\|T\|$ is absolutely continuous (see Theorem \ref{metric_currents}). 

This result has been proved in \cite[Theorem 1.15]{DPR}, relying on results from \cite{S}. Our proof is a direct consequence of the following theorem, of independent interest, which we obtain as a corollary of Theorem \ref{main}. 

\begin{theorem}\label{preiss}
Let $\mu$ be a finite Borel measure on $\R^n$ and assume that $\mu$ is not absolutely continuous with respect to Lebesgue. Then there exists a sequence of continuously differentiable, equi-Lipschitz functions $\{g_h\}_{h\in\N}$ converging pointwise to the identity and such that 
$$\limsup_{h\to\infty}\int_{\R^n}{\rm{det}}(\nabla g_h) d\mu <\mu(\R^n).$$
\end{theorem}

\subsection*{Acknowledgements}
The author 
was supported by the ERC Grant 306247 \emph{Regularity of area minimizing currents}.  

\section{Notation and preliminaries}
We begin this section by introducing some general notations about measures. Then we define the notion of cone-null set (see \cite{ACP3}) and we recall some properties of the decomposability bundle of a measure, defined in \cite{AM}. Lastly we recall a fact from \cite{AM}: a measure is supported on a $C$-null set, for some closed cone $C$, whenever its decomposability bundle intersects $C$ only at the origin.

\subsection{General notation}
Through this note, sets and functions on $\R^n$ are assumed to be Borel measurable, and measures on $\R^n$ are positive, finite, Radon measures on the Borel $\sigma$-algebra, with the obvious exception of the Lebesgue measure $\Leb^n$ and the Hausdorff measures $\Haus^k, (k\leq n)$. We say that a measure $\mu$ on $\R^n$ is supported on the (Borel) set $E$ if $\mu(\R^n\setminus E)=0$. We say that a measure $\mu$ is absolutely continuous with respect to a measure $\nu$, and we write $\mu\ll\nu$, if $\mu(E)=0$ for every Borel set $E$ with $\nu(E)=0$. We say that $\mu$ is singular with respect to $\nu$ if $\mu$ supported on a Borel set $E$ with $\nu(E)=0$. If we do not specify what is the corresponding measure $\nu$, we always implicitly refer to the Lebesgue measure. If $\mu$ is a measure and $E$ is a Borel set, we denote $\mu\trace E$ the measure defined by
$$\mu\trace E(A)=\mu(A\cap E), \quad\mbox{ for every Borel set }A.$$
\subsection{Rectifiable sets}
Given $m=1,2,\dots$, a subset $E \subset \R^n$ is called $m$-\emph{rectifiable} if $\Haus^{m}(E) < \infty$ and $E$ can be covered, except for an $\Haus^{m}$-null subset, by countably many $m$-dimensional surfaces of class $C^{1}$. If $E$ is $m$-rectifiable, then one can define for $\Haus^{m}$-a.e. $x \in E$ a notion of $m$-dimensional \emph{approximate tangent space} to $E$. Such a tangent space will be denoted ${\rm Tan}(E,x)$ and it coincides with the classical tangent space if $E$ is a piece of an $m$-surface of class $C^1$.   

\subsection{Cone-null sets}\label{s:conenull}
For $j=1,\ldots, n$, we introduce the positive closed cones
$$C_j^+:=\{x=(x_1,\ldots,x_n)\in\R^n:x_j\geq 2^{-\frac{1}{2}}|x|\}.$$
For every $j=1,\ldots, n$, we denote also the cones $C_j:=C_j^+\cup(-C_j^+)$. 
Notice that any $k$-tuple ($k\leq n$) of vectors lying in the interior of different cones is linearly independent.

Given a cone $C_j$ we call \emph{$C_j$-curve} any set of the form $G=\gamma(J)$, where $J$ is a compact interval in $\R$ and $\gamma:J\to\R^n$ is Lipschitz and satisfies $\gamma'(s)\in C_j$ for a.e. $s\in J$. It is important to observe that the condition of being a $C_j$-curve is closed under uniform convergence of the corresponding Lipschitz functions (when the curves are parametrized on the same interval $J$). 
Following \cite{ACP3}, we say that a set $E$ in $\R^n$
is \emph{$C_j$-null} if
$$\Haus^1(E\cap G)=0,$$ for every $C_j$-curve $G$, where $\Haus^1$ denotes the 1-dimensional Hausdorff measure.

One of the main tools that we need from \cite{AM} is the following lemma (see \cite[Lemma 7.3]{AM}), which is a corollary of the general result of \cite{R}. We refer the reader to \cite[Section 2.3]{AM} for a formal definition on the notion of integral of a parametrized family of measures.
\begin{lemma}
\label{l-rainwatercor2}
Let $j\in\{1,\dots,n\}$. For every measure $\mu$ on $\R^n$, one of the 
following (mutually incompatible) alternatives 
holds:
\begin{itemizeb}
\item[(i)]
$\mu$ is supported on a Borel set $E$ which 
is $C_j$-null;
\item[(ii)]
there exists a non-trivial measure of
the form $\smash{\mu'=\int_0^1 \, \mu_t\, dt}$
where $\mu'$ is absolutely continuous \wrt $\mu$, 
each $\mu_t$ is the restriction of $\Haus^1$ 
to some $1$-rectifiable set $E_t$, and
\[
\Tan(E_t,x) \subset C_j,
\quad\text{ for $\mu_t$-a.e.~$x$ and a.e.~$t$.}
\]
\end{itemizeb}
\end{lemma}

\subsection{Decomposability bundle of a measure}\label{s:decomp}
In \cite{AM}, to any Radon measure $\mu$ is assigned a Borel map $x\mapsto V(\mu,x)$, called the \emph{decomposability bundle} of the measure $\mu$, which associates to every point $x\in \R^n$ a vector subspace of $\R^n$. Roughly speaking, one constructs such vector subspace writing ``pieces'' of $\mu$ as an integral of a parametrized family of 1-dimensional rectifiable measures and collecting all the corresponding tangential directions at every point. We refer the reader to \cite[Section 2.6]{AM} for the precise definition. Here we recall only a property which we strictly need in the present note. Even if here we state it as a lemma, indeed such property follows from the very definition of decomposability bundle.

\begin{lemma}\label{l:decomp_bund}
Let $\mu$ be a measure on $\R^n$. Assume there exists a non-trivial measure of
the form $\smash{\mu'=\int_0^1 \, \mu_t\, dt}$
where $\mu'$ is absolutely continuous \wrt $\mu$ and
each $\mu_t$ is the restriction of $\Haus^1$ 
to some $1$-rectifiable set $E_t$. Then
\[
\Tan(E_t,x) \subset V(\mu,x),
\quad\text{ for $\mu_t$-a.e.~$x$ and a.e.~$t$.}
\]
\end{lemma}

Combining Lemma \ref{l-rainwatercor2} and Lemma \ref{l:decomp_bund} we immediately get the following proposition.

\begin{proposition}\label{cones}
Let $j\in\{1,\dots,n\}$ and let $\mu$ be a measure on $\R^n$. If $V(\mu,x)\cap C_j=\{0\}$ for $\mu$-a.e. $x\in\R^n$, then $\mu$ is supported on a Borel set $E$, which is $C_j$-null.
\end{proposition}
We remark that the reverse implication holds true as well, nevertheless we will not need this fact in the present note.
\section{Structure  of cone-null sets}
One of the main ideas for the proof of Theorem \ref{main} is borrowed from \cite{ACP1}, where the main result is deduced from a geometric interpretation of the classical combinatorial result of \cite{D}, due to Dilworth (see also \cite{ES}).

\noindent In a partially ordered set $(\mathcal{S},\leq)$, with the term \emph{chain} we denote a totally ordered subset of $\mathcal{S}$. An \emph{antichain} is a subset $S\subset\mathcal{S}$ such that for every $(s,t)\in S\times S$ with $s\leq t$ it holds $s=t$. The following theorem (see \cite{M}) is a dual version of Dilworth's theorem. For the reader's convenience we include its short proof. We denote by $\sharp(S)$ the number of elements of the set $S$.

\begin{theorem}\label{mirsky}
Let $(\mathcal{S},\leq)$, be a partially ordered finite set. Then the maximal cardinality of a chain in $\mathcal{S}$ equals the smallest number of antichains into which $\mathcal{S}$ can be partitioned.
\end{theorem}
\begin{proof}
For every $s\in \mathcal{S}$, let 
$$l(s):=\sup\{\sharp(S): \mbox{ $S$ is a chain and $s$ is a maximal element of $S$}\}.$$ 
Let $L:=\max\{l(s):s\in \mathcal{S}\}$. Clearly, for every $j=1,\ldots,L$, the set
$$A_j=\{s\in \mathcal{S}:l(s)=j\}$$
is an antichain and $$\mathcal{S}=\bigcup_{j=1}^L A_j.$$ It is not possible to find a partition with a smaller number of antichains, since every two elements of the chain of maximal cardinality necessarily belong to different antichains.
\end{proof}

The next proposition follows from the previous theorem, considering on any finite subset of $\R^n$ the partial order induced by the closed cones $C_j^+$. More precisely, for fixed $j\in\{1,\ldots, n\}$ and $\mathcal{S}$ a finite subset of $\R^n$, we introduce the following partial order on $\mathcal{S}:$
\begin{equation}\label{order}
s\leq t \quad\mbox{ if $s,t\in\mathcal{S}$ satisfy $t=s+v$, for some $v\in C_j^+$}.
\end{equation}
A crucial (although elementary) observation regarding such partial order is that every antichain $A$ is the graph of a $1$-Lipschitz function ${f:\pi_jA\subset\{x_j=0\}\to\R}$, where $\pi_j$ is the orthogonal projection onto $\{x_j=0\}$.

\begin{proposition}\label{coveringtheorem}
Let $E$ be a compact set in $\R^n$ which is $C_j$-null. Then for every $\delta>0$, there are (finitely many) piecewise affine, $1$-Lipschitz functions $f_1,\dots,f_N:\R^{n-1}\to\R$, such that

$$E\subset \bigcup_{i=1}^N I^j_{\delta/N}(f_i).$$
\end{proposition}

\begin{proof}
Without loss of generality, we may assume $E \subset [0,1]^n$. For every $k\in\N$, let $G_k$ be the orthogonal grid obtained dividing each side of $[0,1]^n$ into $k$ equal parts. Let $E_k$ be the set of the centers of the cells of $G_k$ which have non-empty intersection with the set $E$. Consider on $G_k$ the partial order defined in \eqref{order}.

Denote by $\ell_k$ the maximal cardinality of a chain in $E_k$. Our first aim is to prove that 
\begin{equation}\label{opiccolo}
\lim_{k\to\infty}\frac{\ell_k}{k}=0,
\end{equation}
in order to deduce from Theorem \ref{mirsky} that $E_k$ can be covered with $o(k)$ antichains. 

Assume by contradiction that there exist $l>0$ such that, for infinitely many indexes $k$, there is a chain $C_k:=(c_1^{k},\ldots,c_{m_k}^{k})$ in $E_k$ of cardinality at least $lk$. For $i=1,\ldots,m_k$, denote
$t_{i}:=c_{i}^{k}\cdot e_j$ and consider a function
$g_k:\{0,t_1,\ldots,t_{m_k},1\}\rightarrow [0,1]^n$, defined by $$g_k(t_i):=c_i^{k},\quad\mbox{ for every $i=1,\ldots,m_k$}$$
and
$$g_k(0):=c_1^k-t_1e_j,\quad g_k(1):=c^k_{m_k}+(1-t_{m_k})e_j.$$
Extend $g_k$ to a curve $\gamma_k:[0,1]\to[0,1]^n$ which is affine on $[0,t_1]$, on $[t_{m_k},1]$ and on $[t_i,t_{i+1}]$ for every $i=1,\ldots,m_{k-1}$. Clearly $\gamma_k([0,1])$ is a $C_j$-curve, and by construction, $\gamma_k$ is $\sqrt{2}$-Lipschitz. Hence, up to a (non-relabelled) subsequence, $\gamma_k$ converges to a Lipschitz function $\gamma$ as $k\to\infty$ and $\gamma(I)$ is a $C_j$-curve, as observed in \S\ref{s:conenull}. We want to show that $\Haus^1(\gamma(I)\cap E)>0$, which would be a contradiction, since $E$ is $C_j$-null. For every $k$ define a function $\phi_k:[0,1]\rightarrow\R$ by
$$\phi_k(t):={\rm{dist}}(\gamma_k(t),E).$$
Since $\gamma_k$ uniformly converges to $\gamma$, then $\phi_k$ uniformly converges to the continuous function 
$$\phi:= t\mapsto{\rm{dist}}(\gamma(t),E).$$ 

Observe that for every $k$ and for every $t\in[0,1]$ such that $\gamma_k(t)$ belongs to a cell of $G_k$ which contains one of the $c^k_i$ it holds $\phi_k(t)\leq k^{-1}\sqrt{n}$. The set $I_k\subset[0,1]$ of such parameters $t$ has length $|I_k|\geq l$, by the contradiction assumption, hence we have $$\phi_k\leq k^{-1}\sqrt{n},\quad \mbox{ on a set of length at least $l$, for every $k$}.$$
Fix now $\varepsilon>0$, and let $k$ be such that $\|\phi_k-\phi\|_\infty\leq \varepsilon$ and $k^{-1}\sqrt{n}\leq\varepsilon$. Then by triangular inequality $\phi\leq 2\varepsilon$ on a set of length at least $l$. This proves that $\phi\equiv0$ on a set of length at least $l$, hence, since $E$ is compact, we have the contradiction that the $C_j$-curve $\gamma(I)$ satisfies 
$$\Haus^1(\gamma(I)\cap E)\geq l.$$

This proves \eqref{opiccolo}. Now by Theorem \ref{mirsky}, $E_k$ can be covered by $\ell_k$ antichains. As we observed after \eqref{order}, every antichain $A$ is the graph of a $1$-Lipschitz function $h_A$, defined on a discrete set contained in $\{x_j=0\}$, with values in $[0,1]$. For every antichain $A$, let $f_A$ be a (piecewise affine) $1$-Lipschitz extension of $h_A$ to $\{x_j=0\}$. The open slab $I^j_{2k^{-1}\sqrt{n}}(f_A)$ of width $2k^{-1}\sqrt{n}$ around $f_A$ contains every cell intersected by the graph of $f_A$. Therefore $E$ can be covered by $\ell_k$ slabs of total width $2\ell_k k^{-1}\sqrt{n}$, which, in view of \eqref{opiccolo}, completes the proof of the proposition.
\end{proof}

\section{Proof of Theorem \ref{main}}
We begin with the following lemma. For $m\leq n$, by $\gamma_{n,m}$ we denote the Haar measure on the Grasmannian $\mbox{Gr}_{n,m}$ of (unoriented) $m$-planes in $\R^n$ (see \cite[Section 2.1.4]{KP}) and by $\sigma$ we denote the Haar measure on the special orthogonal group $SO(n)$. Moreover we denote 
$$S:=\bigcup_{j=1}^n {C_j}.$$
For $n\geq 3$ and for $j=1,\dots,n$ we say that a hyperplane $v\in\mbox{Gr}_{n,n-1}$ is \emph{tangent} to $C_j$ if $C_j\cap v$ is an $(n-2)$-plane. We say that $v$ is tangent to $S$ if it is tangent to $C_j$ for some $j=1,\dots,n$. Notice that if $v$ is not tangent to $C_j$, but $v\cap C_j\neq\{0\}$, then $v$ intersects the interior of $C_j$.

\begin{lemma}\label{lemma} Let $n\geq 3$, let $\mu$ be a finite measure on $\R^n$, and let $V:\R^n\to {\rm{Gr}}_{n,n-1}$ be a Borel map. Then for $\sigma$-almost every rotation $\rho\in SO(n)$ it holds
\begin{equation}\label{e:lemma}
\rho(V(x))\quad\mbox{is not tangent to $S$, for $\mu$-a.e. $x\in\R^n$}.
\end{equation}
\end{lemma}
\begin{proof}
Firstly we observe that $\gamma_{n,n-1}$-almost every $v\in {\rm{Gr}}_{n,n-1}$, is not tangent to $S$. Indeed for every $j\in\{1,\dots,n\}$ the set of $v\in{\rm{Gr}}_{n,n-1}$ which are tangent to $C_j$ has $\gamma_{n,n-1}$-measure zero. In particular for every $v\in{\rm{Gr}}_{n,n-1}$, $\rho(v)$ is not tangent to $C_j$, for $\sigma$-a.e. $\rho\in SO(n)$. Hence, $\rho(v)$ is not tangent to $S$, for $\sigma$-a.e. $\rho\in SO(n)$. 

Now, denote by $f(x,\rho):\R^n\times SO(n)\to\{0,1\}$ the Borel function 
$$f(x,\rho):=
\begin{cases}
1\quad \mbox{if $\rho(V(x))$ is tangent to $S$},\\
0\quad \mbox{otherwise}.
\end{cases}
$$
By Fubini's theorem
$$\int_{x\in\R^n}\int_{\rho\in SO(n)}f(x,\rho)\;d\sigma(\rho)\;d\mu(x)=\int_{\rho\in SO(n)}\int_{x\in\R^n}f(x,\rho)\;d\mu(x)\;d\sigma(\rho).$$
The inner integral in the LHS being zero for every $x$, implies that the inner integral in the RHS is zero for $\sigma$-a.e. $\rho$, which proves the lemma.
\end{proof}
\begin{proof}[Proof of Theorem \ref{main}]
Let $x\mapsto V(\mu,x)$ be the decomposability bundle of the measure $\mu$. Since $\mu$ is singular, by \cite[Corollary 1.12]{DPR} and \cite[Corollary 6.5]{AM} it holds 
\begin{equation}\label{nonfull}
V(\mu,x)\neq\R^n, \quad\mbox {for $\mu$-a.e. $x$}. 
\end{equation}

Firstly we want to prove that, up to a rotation $\rho:\R^n\to\R^n$, the set $E$ of points $x\in\R^n$ such that $V(\mu,x)$ has non trivial intersection with every cone $C_j$, for $j=1,\ldots,n$, has measure $\mu(E)=0$. This is trivial for $n=2$, because $C_1\cap C_2$ is just the union of 2 lines. Let then $n\geq 3$.
 
By \eqref{nonfull} we can find a Borel measurable map $V':\R^n\to \mbox{Gr}_{n,n-1}$, such that
$$V(\mu,x)\subset V'(x),\quad\mbox{for $\mu$-a.e. $x\in\R^n$}.$$
Let $\rho$ be any rotation satisfying \eqref{e:lemma}, where we applied Lemma \ref{lemma} to the map $V:=V'$. For the sake of simplicity we will assume that $\rho$ is the identity map.

Since for $\mu$-a.e. $x$, $V'(x)$ is not tangent to $S$, then by definition of $E$, for $\mu$-a.e. $x\in E$, $V'(x)$ must have non-trivial intersection with the interior of every $C_j$, for $j=1,\dots,n$. Then $\mu(E)=0$, because $\mbox{dim}(V'(x))= n-1$ for every $x\in\R^n$, whereas, as observed in \S\ref{s:conenull}, vectors in the interior of different cones are linear independent.

For $j=1,\ldots, n$, denote $\mu_j:=\mu\trace E_j$, where
\begin{equation}\label{defEj}
E_j:=\{x\in\R^n:V(\mu,x)\cap C_j=\{0\}\mbox{ and }V(\mu,x)\cap C_k\neq\{0\}\mbox{ for }k<j\}.
\end{equation}

Observe that the union over $j$ of the sets $E_j$ covers $\R^n\setminus E$, hence, by the previous discussion, it covers $\mu$-a.e. point of $\R^n$.

Since by definition of $E_j$, for $\mu_j$-a.e. $x$ it holds $V(\mu_j,x)\cap C_j=\{0\}$, then by Proposition \ref{cones}, $\mu_j$ is supported on a $C_j$-null Borel set $F_j$.

The conclusion then follows by decomposing each set $F_j$ as the union of a $\mu$-negligible set and a countable union of compact $C_j$-null sets $\{K_i^j\}_{i\in\N}$ (clearly the property of being $C_j$-null is preserved by subsets) and applying Proposition \ref{coveringtheorem} to each compact set $K_i^j$, choosing the parameter $\delta_i^j$ in the proposition so that $\sum_{i,j}\delta_i^j\leq\delta$.
\end{proof}

\section{Covering with disjoint slabs}
In some circumstances it could be important that the slabs $I_i^j$ in $\mathcal{F}$ of Theorem \ref{main}, corresponding to the same superscript $j$, are disjoint. Moreover it is also possible to require that the corresponding functions $f_i^j$ are of class $C^1$, slightly increasing their Lipschitz constant. We state this result in the following corollary. The complete proof can be found in \cite[Corollary 4.1.3]{Ma} and it is obtained modifying the slabs of Proposition \ref{coveringtheorem} a posteriori. See e.g. the use made in \cite{FK} of this type of covering in the plane, for an interesting application where it is important to have disjoint slabs.

\begin{corollary}\label{corollary}
Let $E$ be a compact set in $\R^n$ which is $C_j$-null and let $\mu$ be a finite Borel measure supported on $E$. Then for every $\varepsilon_0>0$, there exists $\varepsilon\leq\varepsilon_0$ and finitely many $1$-Lipschitz functions $f_1,\dots,f_N:\R^{n-1}\to\R$ such that the slabs $I^j_{\varepsilon/N}(f_1),\dots,I^j_{\varepsilon/N}(f_N)$ are disjoint and satisfy 
$$\mu\left(E\setminus\bigcup_{i=1}^NI_{\varepsilon/N}^j(f_i)\right)=0.$$
\end{corollary}
\begin{remark}
The interested reader is referred to \cite[Proposition 4.1.15]{Ma} for the details on how to make the slabs disjoint and at the same time requiring that the corresponding functions are of class $C^1$. The price to pay is a small increase in the Lipschitz constant.\\ 
\end{remark}
\begin{proof}[Proof of Corollary \ref{corollary}]
{\bf{Step 1: ordering the slabs.}} Let $f_1,\dots,f_N$ be the functions obtained applying Proposition \ref{coveringtheorem} to the set $E$, with $\delta:=\varepsilon_0/2$. Firstly we define a new set of 1-Lipschitz functions $f^1_1,\dots,f^1_N$ such that
\begin{equation}\label{ordered}
f^1_i\leq f^1_j, \mbox{ for } i<j\quad\mbox{ and }\quad E\subset \bigcup_{i=1}^N I^j_{\delta/N}(f^1_i)=\bigcup_{i=1}^N I^j_{\delta/N}(f_i). 
\end{equation}
To get this, we define, for every $x\in\R^{n-1}$, 
$$f^1_i(x):=f_{\sigma(i,x)}(x),$$
where $\sigma(i,x)$ are defined inductively as follows: let
$$\sigma(1,x):=\min\{j:f_j(x)\leq f_k(x), \mbox{ for every }k=1,\dots, N\},$$
and $I_1(x):=\{\sigma(1,x)\}$; moreover, for $i=2,\dots,N$, let
$$\sigma(i,x):=\min\{j\not\in I_{i-1}(x):f_j(x)\leq f_k(x), \mbox{ for every }k\not\in I_{i-1}(x)\},$$
and
$$I_i(x)=\{\sigma(j,x):j\leq i\}.$$
Observe that the first property in \eqref{ordered} follows directly from the definition of $\sigma(i,x)$ and the second property follows from the simple observation that for every $x$ it holds $I_N(x)=\{1,\dots,N\}$.
Moreover, denoting $E^i_j:=\{x:\sigma(i,x)=j\}$, for every $i=1,\dots,N$ it holds
$$f^1_i=\sum_{j=1}^n\chi_{E^i_j}f_j,$$
hence $f^1_i$ is 1-Lipschitz on each $E^i_j$. Moreover $f_k=f^1_i=f_j$ on $\partial E^i_j\cap\partial E^i_k$. This suffices to prove that $f^1_i$ is 1-Lipschitz for every $i=1,\dots,N$.\\

{\bf{Step 2: separating the slabs.}}
Fix $\varepsilon\in[\delta, 2\delta]$ to be chosen later. We define another set of 1-Lipschitz functions $f^2_1,\dots f^2_N$ such that
\begin{equation}\label{separated}
f^2_i\leq f^2_{i+1}-2\varepsilon/N, \quad \mbox{ for } i=1,\dots,N-1\quad\mbox{ and }\mu\left(E\setminus\bigcup_{i=1}^N I^j_{\varepsilon/N}(f^2_i)\right)=0,
\end{equation}
which completes the proof of the corollary. Again, we construct the functions inductively.
Let $f^2_1:=f^1_1$ and for $i=2,\dots,N$ let 
$$f^2_i:=\max\{f^2_{i-1}+2\varepsilon/N,f^1_i\}.$$
The first property of \eqref{separated} holds by definition. Regarding the second property, we observe that $\bigcup_{i=1}^N I^j_{\varepsilon/N}(f^2_i)$ covers the set 
$$E\setminus\bigcup_{i=1}^{N}{\rm{graph}}(f^1_i+\varepsilon/N).$$ 
To conclude, it is sufficient to choose $\varepsilon\in[\delta,2\delta]$ satisfying 
$$\mu\left(\bigcup_{i=1}^{N}{\rm{graph}}(f^1_i+\varepsilon/N)\right)=0.$$
\end{proof}
\section{Flat chain conjecture}\label{currents}
In this section, we assume the reader to be familiar with the work \cite{AK}. We refer to \cite{AK} also for notation and definitions. As an application of Theorem \ref{main}, we provide a simple proof of the following theorem. We remark that the result has been proved already in \cite[Theorem 1.15]{DPR}, using more technical results from \cite{S}.

\begin{theorem}\label{metric_currents}
Let $T\in\mathbf{M}_n(\R^n)$ be top-dimensional Ambrosio-Kirchheim metric current. Then $\|T\|\ll\Leb^n$.
\end{theorem}

As it was observed in the proof of \cite[Theorem 3.8]{AK}, Theorem \ref{metric_currents} is a direct consequence of Theorem \ref{preiss}. For the reader's convenience, we include the short proof of this fact at the end of the section.
The existence of the maps $\{g_h\}_{h\in\N}$ in Theorem \ref{preiss} can be obtained with the clever technique used in \cite[Lemma 4.12]{AM}, which on the other hand is a particular case of a result contained in \cite{ACP3}. Here we show a proof which we find slightly more geometrically transparent, using the slabs given by Theorem \ref{main}.

\begin{proof}[Proof of Theorem \ref{preiss}]
Assume without loss of generality that $\mu$ is supported on $[0,1]^n$. We denote by $\mu_{ac}$ and $\mu_{sing}$ respectively the absolutely continuous and the singular measures given by the Radon Nikod\'ym decomposition of $\mu$ (see \cite[Theorem 2.22]{AFP}). Remember that by assumptions $\mu_{sing}\neq 0$.
Let $\rho$ be the rotation given by Theorem \ref{main} applied to the measure $\mu_{sing}$. Up to a change of coordinates, we can assume that $\rho$ is the identity map. 
For arbitrarily small $\delta>0$ we will construct a smooth $2n$-Lipschitz map $g_\delta:\R^n\to\R^n$ such that, denoting Id$:\R^n\to\R^n$ the identity map, it holds $|g_\delta-{\rm{Id}}|\leq\delta$ and 
\begin{equation}\label{jacobian}
\int_{\R^n} {\det}(\nabla g_\delta) d\mu_{sing}\leq\delta,
\end{equation}
which clearly implies the theorem, by the well-known $w^*$-continuity property of determinants in the Sobolev space $W^{1,\infty}$ (see e.g. \cite{Da}).

Fix $\delta>0$ and for $j=1,\dots,n$, let $E_j$ be the sets defined as in \eqref{defEj}. Observe that, since the decomposability bundle of a measure $\nu$ which is absolutely continuous with respect to $\Leb^n$ coincides with $\R^n$, $\nu$-almost everywhere, we could have used $\mu_{sing}$ in place of $\mu$ in \eqref{defEj}. By \cite[Corollary 1.12]{DPR}, \cite[Corollary 6.5]{AM}, and the previous discussion, it holds $$\mu_{sing}(\R^n)=\sum_{j=1}^n\mu(E_j)>0,$$ and by Proposition \ref{cones}, for every $j$ there is a $C_j$-null compact set $K_j\subset E_j$ such that 
\begin{equation}\label{resto1}
\sum_{j=1}^n\mu_{sing}(E_j\setminus K_j)\leq\frac{\delta}{2(2n)^n}.
\end{equation}
For fixed $j$, let 
$$\mathcal{F}:=\{I_i:=I^j_{\varepsilon/N}(f^j_i)\}_{i\in\{1,\dots,N\}}$$ be the family of disjoint slabs given by Corollary \ref{corollary} applied to the compact set $K_j$ and the measure $\mu\trace K_j$, with $\varepsilon_0:=\delta/(2n)$. 
Denote by $A_j$ the open set 
$$A_j:=\bigcup_{i=1}^NI_{\varepsilon/N}^{j}(f^j_i)$$ and by $F_j$ a compact subset of $A_j\cap K_j$ such that 
\begin{equation}\label{resto2}
\sum_{j=1}^n\mu_{sing}(K_j\setminus F_j)\leq\frac{\delta}{2(2n)^n}.
\end{equation}
Denote by $\eta$ the positive quantity
$$\eta:=\min_j\{{\rm{dist}}(F_j,\R^n\setminus A_j)\}.$$
We denote by $f_j:\R^n\to\R$ the function 
$$f_j(z_1,\dots,z_n):=z_j-\Haus_1(\{x\in A_j:\hat x_j=\hat z_j, x_j\leq z_j\}).$$
We claim that $f_j$ has the following properties, for every $j$:
\begin{itemize}
\item [(i)] $0\leq z_j-f_j(z)\leq \delta/n$, for every $z\in\R^n$;
\item [(ii)] $f_j(z+te_j)=f_j(z)+t$, if the segment $[z,z+te_j]$ is contained in $A_j$;
\item [(iii)] $f_j$ is 2-Lipschitz.
\end{itemize}
Property (i) follows from the fact that the total width of $\mathcal{F}$ is at most $\delta/(2n)$. Property (ii) follows directly from the definition of $f_j$. To check property (iii), observe firstly that, by definition $|f_j(z+te_j)-f_j(z)|\leq|t|$, for every $z$ and for every $t$. 
To estimate the Lipschitz constant of $f_j$ along $e_j^\perp$, fix $w\in e_j^\perp$ and $z\in\R^n$. Assume without loss of generality that $f_j(z+w)\geq f_j(z)$. Hence 
$$\Haus_1(\{x\in A_j:\hat x_j=\hat z_j, x_j\leq z_j\})\geq\Haus_1(\{x\in A_j:\hat x_j=\hat z_j+tw, x_j\leq z_j\}).$$
Let $t$ be the smallest non-negative real number such that
\begin{equation}\label{lip_estimate}
\Haus_1(\{x\in A_j:\hat x_j=\hat z_j, x_j\leq z_j-t\})=\Haus_1(\{x\in A_j:\hat x_j=\hat z_j+tw, x_j\leq z_j\}).
\end{equation}

It holds $t\leq |w|$, because the slabs in $\mathcal{F}$ are disjoint and the corresponding functions are 1-Lipschitz.
By \eqref{lip_estimate} we have $f_j(z-te_j)=f(z+w)-t$. Since $f_j$ is 1-Lipschitz in the direction $e_j$, the previous estimate and the fact that $t\leq |w|$ is sufficient to prove that $f_j$ is 1-Lipschitz along $e_j^\perp$, which concludes the proof of (iii).

Let now $\phi$ be a radial mollifier with support on the ball $B(0,\eta)$ and consider the convolutions $g_j:=f_j*\phi$. Eventually, define $g_\delta:\R^n\to\R^n$, by
$$g_\delta(z):=(g_1(z),\dots,g_n(z)).$$
Observe that $g_\delta$ is smooth and it has the following properties:
\begin{itemize}
\item [(i)'] $|g_\delta-{\rm{Id}}|\leq \delta$;
\item [(ii)'] $\nabla g_\delta (e_j)=0$ on $F_j$;
\item [(iii)'] $g_\delta$ is $2n$-Lipschitz.
\end{itemize}
From the symmetry of $\phi$ with respect to the axis $\{x_j=0\}$ and from (i) it follows that ${0\leq z_j-g_j(z)\leq\delta/n}$, for every $z\in\R^n$ and for every $j=1,\dots,n$, which implies (i)'. Property (ii) and the definition of $\eta$ imply (ii)'. Property (iii)' follows from (iii).  

Combining (ii)', (iii)' and the estimates \eqref{resto1} and \eqref{resto2}, we get \eqref{jacobian}.

\end{proof}

\begin{proof}[Proof of Theorem \ref{metric_currents}]
We define a (signed) measure $\mu$ by
$$\mu(B):=T(\chi_B dx_1\wedge\dots\wedge dx_n), \quad\mbox{ for every $B\subset\R^n$ Borel}, $$
and we let $\mu\trace A +\mu\trace(\R^n\setminus A)$ be the Hahn decomposition of $\mu$. It is sufficient to prove that both positive measures $\mu\trace A$ and $-\mu\trace(\R^n\setminus A)$ are absolutely continuous.
Assume by contradiction that one of the two measures is not absolutely continuous (without loss of generality we assume that such measure is $\mu\trace A$) and let $g_h$ be the sequence obtained applying Theorem \ref{preiss} to $\mu\trace A$. Then by the continuity property \cite[Definition 3.1 (ii)]{AK} and by \cite[(3.2)]{AK} it holds,
$$\mu\trace A(\R^n)=T(\chi_A dx_1\wedge\dots\wedge dx_n)=\lim_{h\to\infty}T(\chi_A d(g_h)_1\wedge\dots\wedge d(g_h)_n)=$$
$$\lim_{h\to\infty}T(\chi_A{\rm{det}}(\nabla g_h) dx_1\wedge\dots\wedge dx_n)\leq \limsup_{h\to\infty}\int_{A}{\rm{det}}(\nabla g_h) d\mu <\mu\trace A(\R^n),$$
which is a contradiction.
\end{proof}
%
%
\bibliographystyle{plain}

\begin{thebibliography}{99}
\setlength{\itemsep}{3pt plus 1pt}
\newcommand{\aut}[1]{\textsc{#1}}


\bibitem{ACP1}
\aut{Alberti, Giovanni; Cs\"ornyei, Marianna; Preiss, David.}
\newblock Structure of null sets in the plane and applications. 
\newblock \textit{European Congress of Mathematics. Proceedings of 
the 4th Congress (4ECM, Stockholm, June 27-July 2, 2004)}, 
pp. 3--22. Edited by A.~Laptev. 
European Mathematical Society (EMS), Z\"urich 2005.

\bibitem{ACP2}
\aut{Alberti, Giovanni; Cs\"ornyei, Marianna; Preiss, David.}
\newblock Differentiability of Lipschitz functions, 
structure of null sets, and other problems. 
\newblock \textit{Proceedings of the international congress of 
mathematicians (ICM 2010, Hyderabad, India, August 19-27, 2010)}. 
Volume 3 (invited lectures), pp.~1379--1394. 
Edited by R.~Bhatia et al. 
\newblock Hindustan Book Agency, New Delhi, 
and World Scientific, Hackensack (New Jersey), 2010.

\bibitem{ACP3}
\aut{Alberti, Giovanni; Cs\"ornyei, Marianna; Preiss, David.}
\newblock Structure of null sets, differentiability of Lipschitz 
functions, and other problems. 
\newblock Paper in preparation. 

\bibitem{AM}
\aut{Alberti, Giovanni; Marchese, Andrea.}
\newblock On the differentiability of Lipschitz functions with respect to measures in the Euclidean space.
\newblock \textit{Geom. Funct. Anal.} 26 (2016), no. 1, 1--66. 

\bibitem{AFP}
\aut{Ambrosio, Luigi; Fusco, Nicola; Pallara, Diego.}
\newblock \textit{Functions of bounded variation and free discontinuity problems.}
\newblock Oxford Mathematical Monographs. The Clarendon Press, Oxford University Press, New York, 2000.

\bibitem{AK}
\aut{Ambrosio, Luigi; Kirchheim, Bernd.}
\newblock Currents in metric spaces.
\newblock \textit{Acta Math.} 185 (2000), no. 1, 1--80. 

\bibitem{B}
\aut{Bate, David.}
\newblock Structure of measures in Lipschitz differentiability spaces. 
\newblock \textit{J. Amer. Math. Soc.} 28 (2015), no. 2, 421--482. 

\bibitem{Da}
\aut{Dacorogna, Bernard.}
\newblock \textit{Direct Methods in the Calculus of Variations.}
\newblock Applied Mathematical Sciences, 78. Springer-Verlag, Berlin, 1989.

%
%



\bibitem{DPMR}
\aut{De Philippis, Guido; Marchese, Andrea; Rindler; Filip.} 
\newblock On a conjecture of Cheeger
\newblock \href{https://arxiv.org/abs/1607.02554}{arXiv:1607.02554.}

\bibitem{DPR}
\aut{De Philippis, Guido; Rindler, Filip.}
\newblock On the structure of $\mathscr{A}$-free measures and applications. 
\newblock \textit{Ann. of Math. (2)} 184 (2016), no. 3, 1017--1039.

\bibitem{D} 
\aut{Dilworth, Robert P.}
\newblock A decomposition theorem for partially ordered sets.
\newblock \textit{Ann. of Math. (2)} 51 (1950), 161--166.

%
%

\bibitem{ES} 
\aut{Erd\"os, Paul; Szekeres, George.}
\newblock A combinatorial problem in geometry. 
\newblock \textit{Compositio Math.} 2 (1935), 463--470. 



%

\bibitem{FK}
\aut{Fischer, Julian; Kneuss, Oliver.} 
\newblock Bi-Sobolev Solutions to the Prescribed Jacobian Inequality in the Plane with $L^p$ Data.
\newblock \href{https://arxiv.org/abs/1408.1587}{arXiv: 1408.1587}.

\bibitem{GP}
\aut{Gigli, Nicola; Pasqualetto, Enrico.} 
\newblock  Behaviour of the reference measure on RCD spaces
under charts.
\newblock \href{https://arxiv.org/abs/1607.05188}{arxiv: 1607.05188}.

%
%
%


\bibitem{CJ}
\aut{Jones, Peter W.}
\newblock Product formulas for measures and applications to analysis and geometry. 
\newblock \textit{Talk given at the conference “Geometric and algebraic structures in mathematics”, Stony
Brook University}, May 2011. 
\newblock Video available at: \href{http://www.math.sunysb.edu/Videos/dennisfest/}{http://www.math.sunysb.edu/Videos/dennisfest/}.

\bibitem{KM}
\aut{Kell, Martin; Mondino, Andrea.} 
\newblock On the volume measure of non-smooth spaces with ricci
curvature bounded below.
\newblock \href{https://arxiv.org/abs/1607.02036}{arXiv:1607.02036}.

\bibitem{KP}
\aut{Krantz, Steven~G.; Parks, Harold~R.}
\newblock \textit{Geometric integration theory.}
\newblock Cornerstones. Birkh\"auser, Boston 2008.

%


\bibitem{Ma}
\aut{Marchese, Andrea.} 
\newblock 
Two applications of the theory of currents (PhD thesis).
\newblock 
\href{http://cvgmt.sns.it/paper/2332/}{http://cvgmt.sns.it/paper/2332/}.

\bibitem{MS}
\aut{Marchese, Andrea; Schioppa, Andrea.} 
\newblock 
Lipschitz functions with prescribed blowups at many points.
\newblock 
\href{https://arxiv.org/abs/1612.05280}{arXiv: 1612.05280}.




\bibitem{M} 
\aut{Mirsky, Leonid.}
\newblock A dual of Dilworth's decomposition theorem.
\newblock \textit{Amer. Math. Monthly} 78 (1971), 876--877.




%
%
%
%

\bibitem{R}
\aut{Rainwater, John.}
\newblock A note on the preceding paper. 
\newblock \textit{Duke Math. J.} 36 (1969), 799--800.

%

\bibitem{S}
\aut{Schioppa, Andrea.}
\newblock Metric currents and Alberti representations.
\newblock \textit{J. Funct. Anal.} 271 (2016), no. 11, 3007--3081.
%
%
%
%

\end{thebibliography}


%
%

\vskip .5 cm

{\parindent = 0 pt\begin{footnotesize}

A.M.
\\
Institut f\"ur Mathematik,
Mathematisch-naturwissenschaftliche Fakult\"at,
Universit\"at Z\"urich\\
Winterthurerstrasse 190,
CH-8057 Z\"urich,
Switzerland
\\
e-mail: {\tt andrea.marchese@math.uzh.ch}

\end{footnotesize}
}
\end{document}